\documentclass{amsart}
\usepackage{amsmath}
\usepackage{amssymb}
\usepackage{graphicx}
\input xy
\xyoption{all}
\usepackage{color}
\usepackage{enumerate}
\newtheorem{theorem}{Theorem}[section]

\newtheorem{lemma}[theorem]{Lemma}

\theoremstyle{definition}

\newtheorem{remark}{Remark}

\newcommand{\M}{{\rm M}}
\newcommand{\I}{{\rm I}}
\newcommand{\E}{{\rm E}}
\newcommand{\GL}{{\rm GL}}
\newcommand{\SL}{{\rm SL}}
\newcommand{\Gal}{{\rm Gal}}
\newcommand{\Char}{{\rm char}}

\begin{document}

	\title[Free subgroups in maximal subgroups]{On free subgroups in maximal subgroups\\ of skew linear groups}
	
	 \author[Bui Xuan Hai]{Bui Xuan Hai}\thanks{This work is funded by Vietnam National Foundation for Science and Technology Development (NAFOSTED) under Grant No. 101.04-2016.18.}
	 \email{ bxhai@hcmus.edu.vn; huynhvietkhanh@gmail.com}
	  \address{Faculty of Mathematics and Computer Science, University of Science, VNU-HCM, 227 Nguyen Van Cu Str., Dist. 5, Ho Chi Minh City, Vietnam.} 
	  
	\author[Huynh Viet Khanh]{Huynh Viet Khanh}
		
	
	\maketitle
	
	\begin{abstract} The study of the existence of free groups in skew linear groups have been begun since the last decades of the 20-th century. The starting point is the theorem of Tits (1972), now often is referred as Tits' Alternative,  stating that every finitely generated subgroup of the general linear group $\GL_n(F)$ over a field $F$ either contains a non-cyclic free subgroup or it is solvable-by-finite. In this paper, we study the existence of non-cyclic free subgroups in maximal subgroups of an almost subnormal subgroup of the general skew linear group over a locally finite division ring.
		
	\end{abstract}

	\section{Introduction}
	
	Let $D$ be a division ring and $n$ a positive integer. The subject of this paper is to study the problem on the existence of non-cyclic free subgroups in maximal subgroups of an almost subnormal subgroup of the general skew linear group $\GL_n(D)$ provided $D$ is locally finite. Recall that  $D$ is \textit{locally finite} if for every finite subset $S$ of $D$, the division subring $F(S)$ of $D$ generated by the set $F\cup S$ is a finite dimensional vector space over $F$, where $F$ is the center of $D$. Let $G$ be an arbitrary group and $H$ a subgroup of $G$. Following Hartley \cite{Hartley_89}, $H$ is \textit{almost subnormal} in $G$ if there exists such a sequence of subgroups 
		$$H=H_0\leq H_1\leq\ldots\leq H_r=G,$$
	where for any $0\leq i\leq r-1$, either $H_i$ is normal in $H_{i+1}$ or $H_i$ is a subgroup of finite index in $H_{i+1}$.  By definition, every subnormal subgroup in a group  is almost subnormal. In \cite[Example~8]{Haz-Wad}, for a given prime power $k=q^m$ and a local field $F$, Hazrat and Wadsworth have constructed a division ring $D$ with center $F$ in which for any odd prime divisor $p$ of $k+1$, there exists a non-normal maximal subgroup of index $p$ in the multiplicative group $D^*$ of $D$. In view of these examples, one can conclude that there exists a various number of division rings containing almost subnormal subgroups which are not subnormal.
	
	The problem of the existence of non-cyclic free subgroups in maximal subgroups of  skew linear groups was studied by several authors. Let $D$ be a division ring and assume that $M$ is a maximal subgroup of the group $\GL_n(D)$. In \cite{free}, Mahdavi-Hezavehi studied this problem for the case when $n=1$ and $D$ is centrally finite. It was proved that either $M$ contains a non-cyclic free subgroup or there exists a maximal subfield $K$ of $D$ such that $K/F$ is a Galois extension, $K^*$ is a normal subgroup of $M$ and $M/K^*\cong \Gal(K/F)$. In \cite{hai-tu}, B. X. Hai and N. A. Tu extended the study for the case when $D$ is locally finite, $M$ is a maximal subgroup of a subnormal subgroup $G$ of the group $\GL_1(D)$. A surprising result obtained in \cite[Theorem 3.4]{hai-tu} states that if $M$ does not contain non-cyclic free subgroups then $D$ must be centrally finite, $M$ must be absolutely irreducible, and there exists a maximal subfield $K$ of $D$ such that $K/F$ is a Galois extension and $\Gal(K/F)\cong M/K^*\cap G$. The general case for $n\geq 1$ and a maximal subgroup $M$ of the group $\GL_n(D)$ over a centrally finite division ring $D$ was considered in \cite{dorbidi2014}, \cite{moghaddam2017} and \cite{mah}. Here, we consider the more general case for $n\geq 1$ and a maximal subgroup $M$ in an almost subnormal subgroup $G$ of the general linear group $\GL_n(D)$ over a locally finite division ring $D$. The new  result obtained in Theorem \ref{theorem_2.9} generalizes all previous results mentioned above.
	
	Throughout this paper,  $F$ and $D^*$ denote the center and the multiplicative group of a division ring $D$ respectively. For a positive integer $n, \M_n(D)$ is the matrix ring of degree $n$ over $D$. We identify $F$ with $F\I_n$ via the ring isomorphism $a\mapsto a\I_n$, where $\I_n$ is the identity matrix of degree $n$. If $S$ is a subset of $\M_n(D)$, then $F[S]$ denotes the subring of $\M_n(D)$ generated by $S\cup F$. Also, in the case $n=1$, we denote by $F(S)$ the division subring of $D$ generated by the set $F\cup S$. If $H$ and $K$ are two subgroups in a group $G$, then $N_K(H)$ denotes the set of all elements $k\in K$ such that $k^{-1}Hk\leq H$, i.e., $N_K(H)=K\cap N_G(H)$. If $A$ is a ring or a group, then $Z(A)$ denotes the center of $A$.
	
	Let $V =D^n= \left\{ {\left( {{d_1},{d_2}, \ldots ,{d_n}} \right)\left| {{d_i} \in D} \right.} \right\}$. For any subgroup $G$ of $\GL_n(D)$, $V$ may be viewed as $D$-$G$ bimodule. Recall that a subgroup $G$ of $\GL_n(D)$ is  \textit{irreducible} (resp. \textit{reducible, completely reducible}) if $V$ is irreducible (resp. reducible, completely reducible) as $D$-$G$ bimodule. If $F[G]=\M_n(D)$, then   $G$ is  \textit{absolutely irreducible} over $D$. If $G$ is irreducible, then $G$ is \textit{imprimitive} if there exists an integer $ m \geq 2$ such that $V =  \oplus _{i = 1}^m{V_i}$ as left $D$-modules and for any $g \in G$ the mapping $V_i \to V_ig$ is a permutation of the set $\{V_1, \cdots, V_m\}$. If $G$ is irreducible and not imprimitive, then $G$ is \textit{primitive}. 
		
\section{Auxiliary  results}



If $S$ is a subset of $D$ and every element of $S$ is algebraic over $F$, then we say that $S$ is \textit{algebraic} over $F$. The following fact is obvious.

\begin{lemma} \label{lemma_2.2} Let $D$ be a  division ring with center $F$ and $G$ be a subgroup of $D^*$. If $F[G]$ is algebraic over $F$, then $F(G)=F[G]$. \null\hfill $\square$
\end{lemma} 

\begin{lemma} \label{lemma_2.3} Let $D$ be a locally finite division ring with center $F$ and $G$ be a subgroup of $\GL_n(D)$. If $G$ contains no non-cyclic free subgroups, then $G$ is (locally solvable)-by-(locally finite).
	
\end{lemma}
\begin{proof} For any finite subset $S\subseteq G$, let $H$ be the subgroup of $G$ generated by $S$. Assume that $K$ is the division subring of $D$ generated by all entries of all matrices in $S$ over $F$, and let $r=[K:F]$. Viewing $H$ as a subgroup of $\GL_r(F)$, by \cite[Corollary 1]{tits}, $H$ is solvable-by-finite. Hence, $G$ is locally  solvable-by-finite. By \cite[3.3.9, p.103]{shirvani-wehrfritz}, $G$ is (locally solvable)-by-(locally finite). 
\end{proof}
\begin{remark}\label{rem:1}
Recall that a field $F$ is \textit{locally finite} if every its finitely generated subfield is finite. Obviously $F$ is locally finite if and only if the prime subfield $P$ of $F$ is finite and $F$ is algebraic over $P$. If $D$ is a locally finite division ring with such a center $F$, then Jacobson's theorem \cite[Theorem 13.11, p.208]{lam} yields  $D=F$. Hence, if $D$ is a non-commutative locally finite division ring, then the center $F$ of $D$ is not locally finite.
\end{remark}
 
 \begin{lemma} \label{lemma_2.4} Let $D$ be a non-commutative locally finite division ring with center $F$ and $G$ be an almost subnormal subgroup of $\GL_n(D)$. If $G\not\subseteq F$, then $F[G]=\M_n(D)$.
\end{lemma}

\begin{proof} For $n>1$, by \cite[Theorem 3.3]{ngoc_bien_hai_17}, $G$ is normal in $\GL_n(D)$. Therefore, $g^{-1}F[G]g\in F[G]$ for any $g\in \GL_n(D)$. In view of \cite[Theorem H]{moghaddam2017}, it follows that $F[G]=M_n(D)$. If $n=1$, then by Lemma \ref{lemma_2.2}, $F[G]$ is a division ring, hence the conclusion holds by \cite[Theorem 1]{bien_deo_hai_17}.
\end{proof}

\begin{theorem} \label{theorem_2.5} Let $D$  be a non-commutative locally finite division ring with center $F$.  Assume that $G$ is an almost subnormal subgroup of $\GL_n(D)$ and $M$ is a non-abelian maximal subgroup of $G$. Then, $M$ contains no non-cyclic free subgroups if and only if $M$ is abelian-by-locally finite. 
\end{theorem}

\begin{proof} Assume that $M$ contains no non-cyclic free subgroups.  We claim that $M$ is irreducible. Indeed, since every subgroup of $D^*$ is obviously irreducible, we can assume $n>1$. According to \cite[Theorem 3.3]{ngoc_bien_hai_17}, $\SL_n(D)\subseteq G$ and $G$ is normal in $\GL_n(D)$. Assume by contradiction that $M$ is reducible. In view of \cite[1.1.1, p.2]{shirvani-wehrfritz}, there exist a matrix $P\in\GL_n(D)$ and some integer $0<m<n$ such that $PM{P^{-1}}\subseteq H$, where $$H=\left({\begin{array}{*{20}{c}}{\GL_m(D)}&*\\0&{\GL_{n - m}(D)}\end{array}}\right) \cap G.$$ The normality of $G$ in $\GL_n(D)$ and the maximality of $M$ imply that $PMP^{-1}$ is a maximal subgroup of $G$. Thus, either $H=G$ or $H=PMP^{-1}$. If $H=G$, then $\SL_n(D)\subseteq H$. This is imposible since $\I_n+\E_{n1}$ belongs to $\SL_n(D)$ but it is not an element of $H$ (here $\E_{n1}$ is the matrix with $1$ in the position $(n,1)$ and $0$ everywhere else). So, we may assume that $H=PMP^{-1}$. Because $M$ contains no non-cyclic free subgroups, so does $H$. Consequently, the group $$\left({\begin{array}{*{20}{c}}{\SL_m(D)}& 0\\0&{\I_{n - m}}\end{array}}\right) \subseteq H,$$ which is a copy of $\SL_m(D)$, contains no non-cyclic free subgroups that is a contradiction to \cite[Theorem 4.3]{ngoc_bien_hai_17}. 
Therefore, $M$ is irreducible as claimed. Now, we shall prove that $M$ is abelian-by-locally finite. Since $M$ contains no non-cyclic free subgroups, by Lemma \ref{lemma_2.3}, $M$ is (locally solvable)-by-(locally finite). Let $N$ be a locally solvable normal subgroup of $M$ such that $M/N$ is locally finite. According to \cite[1.2.4, p.11]{shirvani-wehrfritz}, $M$ is isomorphic to an absolutely irreducible skew linear group. By \cite[Theorem 1.1]{wehrfritz87}, $N$ contains an abelian normal  subgroup $A$ of $M$ such that $N/A$ is locally finite. Applying \cite[14.3.1, p.429]{robinson}, we conclude that $M/A$ is locally finite. Thus, $M$ is abelian-by-locally finite. 

Conversely, suppose that $M$ is  abelian-by-locally finite and contains a non-cyclic free subgroup $G_1:=\langle x,y\rangle$, which is generated by $x$ and $y$. Let $H$ be an abelian normal subgroup of $M$ such that $M/H$ is locally finite. Then $G_1H/H$ is  finite. So, there exist integers $i,j\geq1$ such that $x^i,y^j\in H$. Since $H$ is abelian, $x^iy^j=y^jx^i$ that is impossible because $x$ and $y$ are generators of the free group $G_1$.
\end{proof}

Suppose that $n=rk$, for some $k>1$. Let $S$ be a transitive subgroup of the symmetric group $S_k$. Convert each $\sigma\in S$ to a $k\times k$ permutation matrix $\bar t$ in the usual way, and set $t=\I_r\otimes\bar t$. Let $G$ be a subgroup of $\GL_r(D)$. For each $\bar{g}\in G$, take the matrix $\I_k$ and replace the $(1,1)$ entry by  $\bar{g}$, the other diagonal entries by $\I_r$, and the remaining entries by zero matrices. Call this matrix $g$. Then both $t$ and $g$ are elements of $\GL_n(D)$. The subgroup of $\GL_n(D)$ generated by all such matrices $g$ and $t$ is the wreath product of $G$ and $S$, denoted by $G\wr S$. 

\begin{lemma} \label{lemma_2.6} Let $D$  be an infinite division ring with center $F$. Assume that $G$ is an almost subnormal subgroup of $\GL_n(D)$ and $M$ is an irreducible maximal subgroup of $G$. If $M$ is imprimitive, then $n\geq 2$ and there exist positive integers $r, k$ with $k>1$ such that $n=rk$ and $M$ contains a copy of $\SL_r(D)$. 
\end{lemma} 

\begin{proof} Assume that $M$ is imprimitive. Then, by definition of imprimitivity, we have $n\geq 2$. By \cite[Theorem 3.3]{ngoc_bien_hai_17}, $G$ is a normal subgroup of $\GL_n(D)$ containing $\SL_n(D)$. By \cite[Lemma 5, p.108]{suprunenko}, there exist positive integers $r, k$ with $k>1$ such that $n=rk$ and $M$ is conjugate to a subgroup, say $M_1$, of $H:=\GL_r(D)\wr S_k$. The normality of $G$ in $\GL_n(D)$ yields $M_1\leq G$, hence $M_1$ is also maximal in $G$. Now, $G\cap H$ is a subgroup of $G$ containing $M_1$, so by maximality of $M_1$ in $G$, either $G\leq H$ or $G\cap H=M_1$. If the first case occurs, then the subgroup $\SL_n(D) $ should be contained in $H$. This is impossible because $\I_n+\E_{n1}$ belongs to $\SL_n(D)$ but it is not an element of $H$. Therefore, $G\cap H=M_1$, which implies that $M_1$ contains a copy of $\SL_r(D)$. Consequently, $M$ contains a copy of $\SL_r(D)$.
\end{proof}

\begin{lemma} \label{lemma_2.7} Let $D$  be a locally finite division ring with center $F$. Suppose that $M$ is a primitive subgroup of $\GL_n(D)$ and $N$ is a normal subgroup of $M$. Then $F[N]\cong \M_t(\Delta)$ for some $t\geq1$ and some locally finite dimensional division $F$-algebra $\Delta$.
\end{lemma} 
\begin{proof} Since $M$ is irreducible, $N$ is completely reducible, so by \cite[1.1.12, p.7]{shirvani-wehrfritz}, $F[N]$ is a semisimple artinian ring. Moreover, in view of \cite[Proposition 3.3]{dorbidi2011}, $F[N]$ is a prime ring. Therefore, $F[N] $ is a simple artinian ring. By the Wedderburn-Artin Theorem,  $F[N] \cong \M_{t}(\Delta)$ for some $ t \geq 1$ and some division $F$-algebra $\Delta$. Let $F_0=Z(\Delta)$ and $S$ be any finite subset of $\Delta$. We claim that $\Delta$ is locally finite. To do this, we must show that $\Delta_2=F_0(S)$, the division subring of $\Delta$ generated by $S$ over $F_0$, is finite dimensional over $F_0$. First, we will prove that $[\Delta_2:F_2]<\infty$, where $F_2=Z(\Delta_2)$. Let $D_1$ be the division subring of $D$ generated by all entries of all elements of $S$ over $F$.  Denoting $r:=[D_1:F]$, we have $$F[S]\leq\M_n(D_1)\leq\M_{nr}(F).$$ 
In view of the Amitsur-Levitzki Theorem,  $F[S]$ satisfies the standard polynomial identity. By \cite[Theorem 1]{amitsur}, the division subring $\Delta_1:=F(S)$ of $\Delta$ is finite dimensional over its center $F_1$. Set $m=[\Delta_1:F_1]$. Since $F\subseteq F_0$, we have $\Delta_1\subseteq\Delta_2$. It is easy to check that  $F_2=C_{\Delta_2}(\Delta_1)$ is a field which contains $F_1$ and $F_1=\Delta_1\cap F_2$. By setting $R=\Delta_1\otimes_{F_1}F_2$, we have $Z(R)=F_2$ and $$R\leq (\Delta_1\otimes_{F_1}\Delta_1^{op})\otimes_{F_1}F_2\cong\M_m(F_1)\otimes_{F_1}F_2.$$ 
Note that by \cite[Corollary 6, p.42]{draxl}, both $R$ and $\M_m(F_1)\otimes_{F_1}F_2$ are  central simple $F_2$-algebras, hence $[R:F_2]<\infty$. Since $R$ is simple, the map $$R=\Delta_1\otimes_{F_1}F_2\to \Delta_2:x\otimes y\mapsto xy$$ is an injective homomorphism. Therefore, $R$ is a domain and $\dim_{F_2}R<\infty$, which implies that $R$ is a division ring. Moreover, $R$ contains $F_0$ and $S$, so $R=\Delta_2$ and $k:=[\Delta_2:F_2]<\infty$. 

Next, we will show that $[F_2:F_0]<\infty$. Suppose that $\{a_1,a_2,\dots, a_k\}$ is a basis of $\Delta_2$ over $F_2$ and $S=\{b_1,b_2,\dots,b_l\}$. For $1\leq i,j\leq k$ and $1\leq s\leq l$, write $a_ia_j=c_{ij1}a_1+c_{ij2}a_2+\cdots+c_{ijk}a_k$, $c_{ij\alpha}\in F_2$ and $b_s=d_{s1}a_1+d_{s2}a_2+\cdots+d_{sk}a_k$, $d_{s\gamma}\in F_2$. Set $$\Omega=\{c_{ij\alpha}:1\leq i,j,\alpha\leq k\}\cup\{d_{s\gamma}:1\leq s\leq l, 1\leq \gamma\leq k\}.$$ Let $K$ be the subfield of $F_2$ generated by $\Omega$ over $F_0$, and let $$\Delta_3=\left\lbrace \sum\limits_{i = 1}^k\alpha_ia_i:\alpha_i\in K\right\rbrace.$$ Then, $\Delta_3$ is a domain and $\dim_K\Delta_3<\infty$. Consequenly, $\Delta_3$ is a division ring containing both $S$ and $F_0$, which yields $\Delta_3=\Delta_2$ and $[F_2:K]<\infty$. Since $K$ is finitely generated over $F_0$, so is $F_2$. In view of \cite[Theorem 8.4.1, p.387]{cohn}, $\M_n(D)$ is algebraic over $F$. Thus, $F_2$ is algebraic over $F_0\supseteq F$. It follows that $[F_2:F_0]<\infty$. Finally, the conditions $[\Delta_2:F_2]<\infty$ and $[F_2:F_0]<\infty$ imply that $[\Delta_2:F_0]<\infty$. 
\end{proof}
It is well-known that the existence of maximal subgroups in division rings is one of the difficult open problems in the theory of skew linear groups. The result obtained in the following theorem contains some useful information for the study of this problem. Some related results can be found in  \cite{dbh} and \cite{free}.

\begin{theorem} \label{theorem_2.8} Let $D$  be a non-commutative locally finite division ring with center $F$. Assume that $G$ is an almost subnormal subgroup of $\GL_n(D) $ and $M$ is a non-abelian maximal subgroup of $G$. If  $M$  is radical over $F$ (that is, for every element $x\in M$, there exists an integer $n(x)$ depending on $x$ such that $x^{n(x)}\in F$), then  $[D:F] < \infty$ and $M$ is abelian-by-finite. 
\end{theorem}

\begin{proof} For any $x\in M'$, assume that $x=[x_1,y_1]^{z_1}\ldots [x_k,y_k]^{z_k}$, where $z_i$'s are integers, and $x^m=a\in F$ for some integer $m\geq 1$. Let $D_1$ be the division subring of $D$ generated by all entries of  $x_1, y_1, \cdots, x_k, y_k $ over $F$. Since $D$ is locally finite, $D_1$ is finite dimensional over $F$. By setting $t=[D_1:F]<\infty$ and viewing $x$ as an element of $\M_n(D_1)$, we have $$1=N_{\M_n(D_1)/F}(x)^m=N_{\M_n(D_1)/F}(x^m)=N_{\M_n(D_1)/F}(a)=a^{n^2t},$$ so $x^{mtn^2}=1$. Therefore $M'$ is torsion. Because $D$ is a locally finite division ring, $M'$ is a locally  linear group, that is, every finite subset of $M'$ generates a subgroup which is isomorphic to a linear group over a field. According to Schur's theorem \cite[Theorem 9.9', p.146]{lam}, $M'$ is locally finite. 

In the proof of Theorem \ref{theorem_2.5} we see that $M$ is irreducible. Therefore, according to \cite[1.2.4, p.11]{shirvani-wehrfritz}, $F[M]\cong \M_{t_1}(\Delta_1)$ for some $t_1\geq1$ and some locally finite dimensional division $F$-algebra $\Delta_1$. The maximality of $M$ in $G$ yields either $F[M]^*\cap G=M$ or $G\subseteq F[M]^*$.  If $F[M]^*\cap G=M$, then $M$ is a non-abelian almost subnormal subgroup of $\GL_{t_1}(\Delta_1)$ that is radical over $F$. This contradicts to \cite[Theorem 4.3]{ngoc_bien_hai_17}, which forces $G \subseteq F[M]$.  Applying Lemma \ref{lemma_2.4}, we conclude that $F[M]=\M_n(D)$. Hence, $Z(M)=M\cap F$.

We claim that $M$ is primitive. Indeed, if $M$ is imprimitive, then according to Lemma \ref{lemma_2.6}, $M$ contains a copy of $\SL_r(D)$ for some integer $r\geq 1$. By \cite[Theorem~4.3]{ngoc_bien_hai_17}, $M$ contains a non-cyclic free subgroup, which contradicts to the fact that $M$ is radical over $F$. Thus $M$ is primitive as claimed. By Lemma \ref{lemma_2.7},  $F[M'] \cong \M_{t_2}(\Delta_2)$ for some $ t_2 \geq 1$ and some locally finite dimensional division $F$-algebra $\Delta_2$. Since $M$ is maximal in $G$ and $M \subseteq N_G(F[M']^*) \subseteq G$, either $M=N_G(F[M']^*)$ or $N_G(F[M']^*)=G$. Let us consider two possible cases:

\bigskip
 
\textit{Case 1: $M=N_G(F[M']^*)$}  

\bigskip

In this case, $G\cap F[M']^*$ is an almost subnormal subgroup of $\GL_{t_2}(\Delta_2)$ contained in $M$. By Remark \ref{rem:1}, $F$ is not a locally finite field, so \cite[Theorem 4.3]{ngoc_bien_hai_17} yields  $M'\subseteq G\cap F[M']^*\subseteq Z(\Delta_2)$. Consequenly, $M$ is metabelian, hence, it is solvable. Since $M$ is irreducible, its unique maximal unipotent normal subgroup $u(M)=1$ (see \cite[2.4]{wehrfritz86}). We now claim that there exists an $FC$-element $x\in M\backslash Z(M)$. For, if $M$ is nilpotent, then by \cite[3.3.1, p.95]{shirvani-wehrfritz}, every element of $M\backslash Z(M)$ is $FC$-element. Otherwise, again by \cite[3.3.1, p.95]{shirvani-wehrfritz}, every element of $M'\backslash Z(M)$ is $FC$-element. If  $A=Core_M(C_M(x))$, then $A$ is a normal subgroup of finite index in $M$. By Lemma \ref{lemma_2.7}, $F[A]\cong \M_{t_3}(\Delta_3)$ for some $t_3\geq1$ and some locally finite dimensional division $F$-algebra $\Delta_3$. Now, the condition $M\subseteq N_G(F[A]^*)\subseteq G$ implies that either $N_G(F[A]^*)=M$ or $N_G(F[A]^*)=G$. If $N_G(F[A]^*)=M$, then $G\cap F[A]^*$ is an almost subnormal subgroup of $\GL_{t_3}(\Delta_3)$ contained in $M$. Again, \cite[Theorem~4.3]{ngoc_bien_hai_17} forces $A\subseteq G\cap F[A]^*\subseteq Z(\Delta_3)$. In orther words, $M$ is abelian-by-finite. Now, in view of \cite[Lemma 11, p.176]{passman_77}, the group ring $FM$ is a PI-ring. Viewing $F[M]$ as a homomorphic image of $FM$, it follows that $F[M]=\M_n(D)$ is also a PI-ring. By Kaplansky's theorem (\cite[Theorem 1]{kap}), $[D:F]<\infty$.  In the case $N_G(F[A]^*)=G$, the group $G\cap F[A]^*$ is almost subnormal in $\GL_n(D)$. If $G\cap F[A]^*$ is contained in $F$, then $A$ is abelian. Consequently, $M$ is abelian-by-finte, and again $[D:F]<\infty $. If $G\cap F[A]^*\not\subseteq F$, then Lemma \ref{lemma_2.4} gives $F[A]=\M_n(D)$. Therefore, all elements of $\M_n(D)$ commute with $x$, or equivalently, $x\in F\cap M=Z(M)$, a contradiction.

\bigskip
 
\textit{Case 2: $N_G(F[M']^*)=G$}
	
\bigskip
	
In this case, $G\cap F[M']^*$ is almost subnormal in $\GL_n(D)$. If $G\cap F[M']^*\subseteq F$, then $M'$ is abelian. By the same argument used in Case 1, it follows that  $[D:F]<\infty$ and $M$ is abelian-by-finite. Assume that $G\cap F[M']^*$ is not contained in $F$.  By Lemma \ref{lemma_2.4}, it follows  that $F[M']=\M_n(D)$.  Now, let us consider the following two subcases:

\bigskip
  
\textit{Subcase 2.1: $\Char(F)=0$} 

\bigskip

By \cite[2.5.14, p.79]{shirvani-wehrfritz}, $M'$ contains a characteristic metabelian subgroup $H$ of finite index in $M'$. Consequently, $H'$ is an abelian normal subgroup of $M$. In view of Lemma \ref{lemma_2.7}, $K:=F[H']$ is a field that is algebraic over $F$. 

Firstly, we assume that $H'\not\subseteq F$, then there exists an element $y\in H'\backslash Z(M)$ because $Z(M)=M\cap F$ as we have proved above. Thus, the elements of the set $y^M:=\{m^{-1}ym\vert m\in M\}\subset K$ have the same minimal polynomial over $F$. This implies $|y^M|<\infty$, so $y$ is an $FC$-element, and consequently, $[M:C_M(y)]<\infty$. Setting $B=Core_M(C_M(y))$, then $B\unlhd M$ and $[M:B]$ is finite. By Lemma \ref{lemma_2.7}, $F[B]\cong\M_{t_4}(\Delta_4)$ for some $t_4\geq1$ and some locally finite dimensional division $F$-algebra $\Delta_4$. The maximality of $M$ in $G$ implies that either $N_G(F[B]^*)=M$ or $N_G(F[B]^*)=G$. If the first case occurs, then $G\cap F[B]^*$ is almost subnormal in $\GL_{t_4}(\Delta_4)$ contained in $M$. By \cite[Theorem 4.3]{ngoc_bien_hai_17}, $B\subseteq G\cap F[B]^*\subseteq Z(\Delta_4)$. Consequently, $M$ is abelian-by-finite and $[D:F]<\infty$. If $N_G(F[B]^*)=G$, then $G\cap F[B]^*$ is almost subnormal in $\GL_n(D)$. If $G\cap F[B]^*$ is contained in $F$, then $B$ is abelian. Consequently, $M$ is abelian-by-finite and $[D:F]<\infty$. If $G\cap F[B]^*\not\subseteq F$, then by Lemma \ref{lemma_2.4}, $F[B]=\M_n(D)$. This implies that all elements of $\M_n(D)$ commute with $y$, or equivalently, $y\in F\cap M=Z(M)$, a contradiction. 

Now, assume that $H'\subseteq F$. In this case, $H$ is nilpotent and locally finite. According to  \cite[2.5.2, p.73]{shirvani-wehrfritz}, $H$ contains an abelian normal subgroup of finite index. Thus,  $M'$ is abelian-by-finite. Again by  \cite[Lemma 11, p.176]{passman_77}, $F[M']=\M_n(D)$ is a PI-ring, which shows that $[D:F]<\infty$. Now, $M$ may be viewed as a linear group containing no non-cyclic free subgroups. By Tits' theorem \cite{tits}, $M$ is solvable-by-finite. Let $A$ be a solvable normal subgroup of finite index in $M$. Applying  \cite[Theorem 6.4, p.111]{dixon}, we deduce that $A$ is abelian-by-finite. Consequently, $M$ is also abelian-by-finite. 

\bigskip
    
\textit{Subcase 2.2: $\Char(F)=p>0$} 

\bigskip
Let $\mathbb{F}_p$ be the prime subfield and $\tilde{F}$ be a maximal locally finite subfield of $F$. By \cite[1.1.14, p.9]{shirvani-wehrfritz}, $\tilde{F}[M']\cong \M_l(S)$ for some $l\geq1$ and some division ring $S$. Since $M'$ is locally finite, $S$ is algebraic over $\tilde{F}$, so $S$ is algebraic over $\mathbb{F}_p$.  By Jacobson's theorem, $S$ is a field. Moreover, $$S=Z(\tilde{F}[M'])\subseteq Z(F[M'])=Z(\M_n(D))=F.$$ 
Now, we have $S=\tilde{F}$ because $S$ is a locally finite subfield of $F$ containing $\tilde{F}$. Therefore, $\tilde{F}[M']\cong \M_l(\tilde{F})$, so $[\tilde{F}[M']:\tilde{F}]=l^2$. By the same argument as in the case $F[M']=\M_n(D)$, we deduce that $F[\tilde{F}[M']]=\M_n(D)$. We claim that $[M_n(D):F]\leq l^2$. Indeed, suppose that $\{u_1,\cdots,u_{l^2}\}$ is a basis for $\tilde{F}[M']$ over $\tilde{F}$. Any element $x\in \M_n(D)$ may be written in the form $$\begin{gathered}
  x=f_1{x_1}+\cdots f_kx_k ,\;\;f_i\in F, x_i\in \tilde{F}[M'] \hfill \\
   \;\;\;= f_1(\alpha_{11}u_1+\cdots+\alpha_{1l^2}u_{l^2})+\cdots+f_k(\alpha _{k1}u_1+\cdots+\alpha_{kl^2}u_{l^2}) \;\ \hfill \\
   \;\;\;= (f_1\alpha_{11}+\cdots+f_k\alpha_{k1})u_1+\cdots+(f_1\alpha_{1l^2}+\cdots+f_k\alpha_{kl^2})u_{l^2}, \alpha_{ij}\in\tilde{F}. \hfill \\ 
\end{gathered} $$ Since $\tilde{F}\subseteq F$, the last equation implies that the elements  $u_1,\cdots,u_{l^2}$ span $\M_n(D)$ over $F$. Therefore, $[M_n(D):F]\leq l^2$ as claimed, and thus $[D:F]\cdot n^2\leq l^2$. The condition $M_l(\tilde{F})\subseteq \M_n(D)$ together with \cite[1.1.9, p.5]{shirvani-wehrfritz} shows that $l\leq n$. Hence, the inequality $[D:F]\cdot n^2\leq l^2$ implies that $n=l$ and $D=F$, which is a contradiction.
\end{proof}

\section{Main result}

Now, we are ready to prove the main result of this paper.

\begin{theorem} \label{theorem_2.9} Let $D$  be a non-commutative locally finite division ring with center $F$. Assume that $G$ is an almost subnormal subgroup of $\GL_n(D) $ and $M$ is a non-abelian maximal subgroup of $G$. If $M$ contains no non-cyclic free subgroups then $[D:F]<\infty$, $F[M]=\M_n(D)$, and there exists a maximal subfield $K$ of $\M_n(D)$  containing $F$ such that $K/F$ is a Galois extension, $N_G(K^*)=M $, $K^*\cap G\unlhd M$, $M/K^*\cap G\cong\Gal(K/F)$ is a finite simple group, and $K^*\cap G$ is the Fitting subgroup of $M$.
\end{theorem}

\begin{proof} If $M$ contains no non-cyclic free subgroups, then by Theorem \ref{theorem_2.5}, $M$ is abelian-by-locally finite. Let $A$ be a maximal subgroup of $M$ with respect to the property: $A$ is an abelian  normal  subgroup of $M$ such that $M/A$ is locally finite.  As we have seen in the proof of Theorem \ref{theorem_2.5}, $M$ is irreducible. If $M$ is imprimitive, then by Lemma \ref{lemma_2.6}, $M$ contains a copy of $\SL_r(D)$ for some integer $r\geq 1$. In view of \cite[Theorem 4.3]{ngoc_bien_hai_17}, $M$ contains a non-cyclic free subgroup, a contradiction. Hence, $M$ is primitive. Since  $M$ is irreducible,  $A$ is completely reducible. Consequently, by \cite[1.1.12, p.7]{shirvani-wehrfritz}, $F[A]$ is a semisimple artinian ring.  The Wedderburn-Artin Theorem implies that
$$ F[A] \cong \M_{n_1}(D_1)\times \M_{n_2}(D_2)\cdots\times \M_{n_s}(D_s),$$
where $D_i$ are division $F$-algebras, $1\leq i\leq s$. Since $F[A]$ is abelian, $n_i = 1$ and $K_i:=D_i=Z(D_i)$ are fields that contain $F$ for all $i$. Therefore, 
$$F[A]\cong K_1\times K_2\cdots\times K_s.$$
Since $M$ is primitive, in view of   \cite[Proposition 3.3]{dorbidi2011}, $F[A]$ is an integral domain, so $s=1$. Hence, $K:=F[A]$ is a subfield of $\M_n(D)$ containing $F$. By Lemma \ref{lemma_2.7}, $F[M]\cong\M_{t_1}(\Delta_1)$ for some $t_1\geq 1$ and some locally finite dimensional division $F$-algebra $\Delta_1$. Since $M$ is maximal in $G$, either $F[M]^*\cap G=M$ or $G\subseteq F[M]^*$. If the first case occurs, then $M$ is a non-abelian almost subnormal subgroup of $\GL_{t_1}(\Delta_1)$ containing no non-cyclic free subgroups that  contradicts to \cite[Theorem~4.3]{ngoc_bien_hai_17}. So, we may  assume that $G\subseteq F[M]^*$. By Lemma \ref{lemma_2.4}, $F[M]=\M_n(D)$, hence,  $Z(M)=M\cap F^*$. Now, we consider two possible  cases: 

\bigskip

\textit{Case 1: $A$ is not contained in $F$}

\bigskip

We claim that $D$ is centrally finite. Take some element $\alpha\in A\backslash F$. The elements of the set $\alpha^M:=\{m^{-1}\alpha m\vert m\in M\}\subset A$ have the same minimal polynomial over $F$. This implies that $|\alpha^M|<\infty$, so $\alpha$ is an $FC$-element, and consequently, $[M:C_M(\alpha)]<\infty$. Setting $$L_1=F(\alpha^M), N=Core_M(C_M(\alpha)), H_1=C_{\M_n(D)}(L_1),$$ then $L_1$ is a subfield of $K$ properly containing $F$, $A\unlhd N\unlhd M$, $N\leq H_1^*$, and $[M:N]$ is finite. Since $H_1=C_{\M_n(D)}(\alpha^M)$, in view of \cite[Proposition 3.3]{dorbidi2011}, $H_1\cong \M_{t_2}(\Delta_2)$ for some $t_2\geq 1$ and some locally finite dimensional division $F$-algebra $\Delta_2$. The subgroup $M$ normalizes $L_1^*$, so it also normalizes $H_1^*$. Therefore,  $M\subseteq N_G(H_1^*)\subseteq G$. Let $B=H_1^*\cap G$. By the maximality of $M$ in $G$, either $G=N_G(H_1^*)$ or $M=N_G(H_1^*)$. Assume that the first case occurs. Then, we have 
$$\alpha\in B \unlhd N_G(H_1^*)=G.$$
This implies that $B$ is a non-central almost subnormal subgroup in $\GL_n(D)$. Therefore $H_1=F[B]=\M_n(D)$ by Lemma \ref{lemma_2.4}. From this, we conclude that $L_1\subseteq F$, a contradiction. Therefore, we may assume that $M=N_G(H_1^*)$, and thus $B=H_1^*\cap G\unlhd M$. The conditions $N\leq B$ and $[M:N]<\infty$ imply that $[M:B]<\infty$. Moreover, $B$ is an almost subnormal subgroup of $H_1^*=\GL_{t_2}(\Delta_2)$ containing no non-cyclic free subgroups, so in view of \cite[Theorem 4.3]{ngoc_bien_hai_17}, we conclude that $B \subseteq Z(\Delta_2)$. In other words, $M$ is abelian-by-finite. Consequently, $F[M]=\M_n(D)$ is a PI-ring, and by Kaplansky's theorem,  $[D:F]<\infty$.

The last paragraph shows that $A\subseteq B$ and $M/B$ is finite. By the maximality of $A$ in $M$, it follows that $A=B$. Since $M\subseteq N_G(F[A]^*)\subseteq G$ and $M$ is maximal in $G$, either $N_G(F[A]^*)=G$ or $N_G(F[A]^*)=M$. In the first case, $F[A]^*\cap G$ is an abelian almost subnormal subgroup of $\GL_n(D)$, so $A\subseteq F[A]^*\cap G\subseteq F$ by \cite[Theorem 4.3]{ngoc_bien_hai_17}, contradiction with case 1. In the second case, $F[A]^*\cap G\unlhd M$. Moreover, since $A\subseteq F[A]^*\cap G$ and $M/A$ is finite, $M/F[A]^* \cap G  $ is also finite. In short, we have ${N_G}(K^*) = M$, $K^* \cap G \unlhd M$ and $M/ K^* \cap G$ is finite. Furthermore,  since $K^* \cap G $ is abelian, by maximality of $A$ in $M$, we conclude that $A=K^*\cap G $. Recall that $A=B$ and $B=H_1^*\cap G$. From this, one has $K^*\cap G=H_1^*\cap G$. Setting $H=C_{\M_n(D)}(K)$, then $K\subseteq H\subseteq H_1$, and thus $K^*\cap G=H^*\cap G$.

Recall that $A$ is normal in $M$, so for any $a\in M$,  the mapping $\theta_a:K\to K$ given by $\theta_a(x)=axa^{-1}$ is well defined. It is clear that $\theta_a$ is an $F$-automorphism of $K$. Thus, the mapping $$\psi:M\to \Gal(K/F)$$ defined by $\psi(a)=\theta_a$ is a group homomorphism with $$\mathrm{ker}\psi=C_M(K^*)=C_{\M_n(D)}(K)^*\cap M=K^*\cap G.$$ Since $F[M]=\M_n(D)$, it follows that $C_{\M_n(D)}(M)=F$. Therefore, the fixed field of $\psi(M)$ is $F$. From this, we conclude that $\psi$ is a surjective homomorphism, and $K/F$ is a Galois extension.  Hence, $M/K^*\cap G\cong\Gal(K/F)$ is a finite group. 

Now, we show that $K$ is a maximal subfield of $\M_n(D)$. To see this, suppose that $\{x_1,\ldots,x_m\}$ is a transversal of $K^*\cap G$ in $M$. Since $M$ normalizes $K^*$,  $$R=x_1K+x_2K+\cdots+x_mK$$  is a subring of $M_n(D)$ containing both $F$ and $M$. Therefore, $R=\M_n(D)$ and $$[\M_n(D):K]_r=[R:K]_r\leq|M/K^*\cap G|=|\Gal(K/F)|=[K:F].$$ By Centralizer Theorem \cite[(vii), p.42] {draxl}, $[M_n(D):K]_r=[H:F]$. Now, the conditions $[H:F]\leq [K:F] $ and $K \subseteq H$ imply $H=K$. In orther words, $K$ is a maximal subfield of $\M_n(D)$.

To prove the simplicity of $M/K^*\cap G$, suppose that  $M_1$ is a normal subgroup of $M$ strictly containing $K^*\cap G$, i.e.
$$K^*\cap G\varsubsetneq M_1 \unlhd M.$$
We have to show $M_1=M$. Indeed, by setting $S=F[M_1]=K[M_1]$, in view of Lemma \ref{lemma_2.7}, we have $S\cong \M_{t_3}(\Delta_3)$ for some $t_3\geq 1$ and some locally finite dimensional division $F$-algebra $\Delta_3$. The condition $M\subseteq N_G(S^*)\subseteq G$ and the maximality of $M$ imply that either $M= N_G(S^*)$ or $N_G(S^*) = G$. Observe that $M_1$ is a non-abelian subgroup contained in $S^*\cap G$. Hence, if the first case occurs, then $S^*\cap G$ is a non-abelian almost subnormal subgroup of $S^*=\GL_{t_3}(\Delta_3)$ containing no non-cyclic free subgroups that contradicts to \cite[Theorem 4.3]{ngoc_bien_hai_17}. This contradiction forces $N_G(S^*)=G$, so $S=\M_n(D)$ by Lemma~ \ref{lemma_2.4}. In other words, we have $K[M]=K[M_1]=\M_n(D)$. Suppose that $\{y_1,\ldots,y_{v}\}$ is a transversal of $K^*\cap M_1$ in $M_1$. Since $M_1$ normalizes $K^*$, $Q=y_1K+\cdots+y_{v}K$ is a subring of $\M_n(D)$ containing both $K$ and $M_1$, so  $Q=\M_n(D)$. In view of the inclusions $K^*\cap G\subseteq M_1\subseteq M\subseteq G$, we have $K^*\cap G\subseteq K^*\cap M_1\subseteq K^*\cap G$,  which implies $K^*\cap M_1=K^*\cap G.$
Hence,
$$[K:F]=[\M_n(D):K]_r=[Q:K]_r\leq|M_1/K^*\cap M_1|\leq|M/K^*\cap G|=[K:F].$$ 
Consequently, $M_1/K^*\cap M_1=M/K^*\cap G$, so $M_1=M$. 

To see that $K^*\cap G$ is the Fitting subgroup of $M$, it suffices to show that this group
is a maximal nilpotent normal subgroup of $M$. Suppose that $K^*\cap G \subseteq  M_2 \subseteq M$ is nilpotent normal subgroup of $M$. By the simplicity of $M/K^*\cap G$, we have either $K^*\cap G = M_2$ or $M_2 = M$. If the first case occurs, then we are done since  $K^*\cap G$ is a maximal abelian normal subgroup of $M$. Now, suppose that $M_2 = M$. Since $[D:F]<\infty$, $M$ may be viewed as a completely reducible nilpotent linear group. By a result in \cite[Corollary 6.5, p.114]{dixon}, $[M:Z(M)]<\infty$. Let $\{z_1,\ldots,z_l\}$ be a tranversal of $Z(M)$ in $M$. Take $x\in G\setminus M$ and set $H=\langle z_1,\ldots,z_l,x \rangle$. Since $M$ is maximal in $G$, it follows $G=HZ(M)$. Recall that $F[M]=\M_n(D)$, from which it follows $Z(M)=M\cap F^*$, so $Z(M)\subseteq F^*$. Now, we have $G'=H'\subseteq H$ because $G=HZ(M)$. Consequenly, $H$ is normal in $G$, so $H$ is a finitely generated almost subnormal subgroup of $\GL_n(D)$. In view of \cite[Theorem 5.4]{ngoc_bien_hai_17}, $H$, and hence $M$ is abelian, a contradiction. 

\bigskip

\textit{Case 2: $A$ is contained in $F$}

\bigskip

If this is the case, then $M$ is radical over $F$. By Theorem \ref{theorem_2.8}, $[D:F]<\infty$ and $M$ is abelian-by-finite. Let $C$ be a maximal abelian normal subgroup of finite index in $M$. Note that $Z(M)=M\cap F^*$, so, if $C\subseteq F$, then $[M:Z(M)]<\infty$. By the same argument used in the last paragraph in Case 1, we conlucde that $M$ is abelian, a contradiction. We may therefore assume that $C$ is not contained in $F$. Replacing $A$ by $C$ in Case 1, we get the result.

\end{proof}


\begin{thebibliography}{}
	
		\bibitem{amitsur} Amitsur S. A., On rings with identities, \textit{J. London Math. Soc.}, \textbf{30} (1955) 464-470 .
		
		\bibitem{bien_deo_hai_17}  Bien M. H.,  Deo  T. T.,  Hai B. X., On a division subring normalized by an almost subnormal subgroup in division rings, \textit{arXiv:1801.01271 [math.RA]} 4 Jan 2018. 
		
		\bibitem{cohn} Cohn P. M. \textit{Skew Fields: Theory of General Division Rings} (Encyclopedia of Mathematics, Vol. 57, Cambridge University Press, Cambridge, 1995).
		
		\bibitem{dbh}  Deo T. T.,  Bien M. H.,  Hai B. X., On radicality of maximal subgroups in $\GL_n(D)$, \textit{J. Algebra}, \textbf{365} (2012) 42-49. 
		
		\bibitem{dixon}  Dixon  J. D., \textit{The Structure of Linear Groups} (Van Nostrand, 1971).
		
		\bibitem{dorbidi2014}  Dorbidi H. R.,  Fallah-Moghaddam R., and Mahdavi-Hezavehi  M., Existence of Nonabelian Free Subgroups in the Maximal Subgroups of $\GL_n(D)$, \textit{Algebra Colloq.}, \textbf{21}(3) (2014) 483-496.
		
		\bibitem{dorbidi2011}  Dorbidi H. R.,  Fallah-Moghaddam R., and Mahdavi-Hezavehi M., Soluble maximal subgroups in GLn(D), \textit{J. Algebra Appl.}, \textbf{10}(6) 1371-1382.  
		
		\bibitem{draxl}  Draxl P., \textit{Skew Fields}, London Math. Soc. Lecture Note Ser., \textbf{81} (Cambridge University Press, 1983).
		
		\bibitem{moghaddam2017}  Fallah-Moghaddam R.,  Mahdavi-Hezavehi M., Free subgroups in maximal subgroups of $\GL_n(D)$, \textit{Commun. Algebra}, \textbf{45}(9) (2017) 3724-3729.
		
		\bibitem{hai-tu}  Hai B. X. and Tu N. A., On multiplicative subgroups in division rings, \textit{J. Algebra  Appl.}, \textbf{15}(3) (2016) 1650050 (16 pages).
				
		\bibitem{Hartley_89}  Hartley B., Free groups in normal subgroups of unit groups and arithmetic groups, \textit{Contemp. Math.}, \textbf{93} (1989) 173-177. 
		
		\bibitem{Haz-Wad}  Hazrat R. and  Wadsworth A.R., On maximal subgroups of the multiplicative group of a division algebra, \textit{J. Algebra}, \textbf{322}(7) (2009) 2528-2543.
		
		\bibitem{kap}  Kaplansky I., Rings with a polynomial identity,  \textit{Bull. Amer. Math. Soc.}, \textbf{54} (1948) 575-580.
		
		\bibitem{lam}  Lam T.Y., \textit{A First Course in Noncommutative Rings}, 2nd edn, GTM \textbf{131} (Springer-Verlag, New York, 2001).
		
		\bibitem{free}  Mahdavi-Hezavehi  M., Free subgroups in maximal subgroups of $\GL_1(D)$,  \textit{J. Algebra}, \textbf{241} (2001) 720-730.
		
		\bibitem{mah}   Mahdavi-Hezavehi M., Tits Alternative for maximal subgroups of $\GL_n(D)$, \textit{J. Algebra}, \textbf{271} (2004) 518-528. 
		
		\bibitem{ngoc_bien_hai_17}  Ngoc N. K.,  Bien M. H.,  Hai B. X., Free subgroups in almost subnormal subgroups of general skew linear groups, \textit{Algebra i Analiz} \textbf{28}(5) (2016) 220-235, translation in \textit{St. Petersburg Math. J.}, \textbf{28}(5) (2017), 707-717.
		
		\bibitem{passman_77} Passman, D. S., \textit{The algebraic structure of group rings} (New York: Wiley- Interscience Publication, 1977).
		
		\bibitem{robinson} Robinson Derek J. S., \textit{A Course in the Theory of Groups} (2nd edn, Springer, 1995).
		
		\bibitem{scott}  Scott W. R., \textit{Group Theory} (Dover Publication, INC, 1987).
		
		\bibitem{shirvani-wehrfritz} Shirvani M.  and Wehrfritz B. A. F., \textit{Skew Linear Groups} (Cambridge Univ. Press, 1986).
		
		\bibitem{suprunenko}  Suprunenko  D.A., \textit{Matrix Groups} (Amer. Math. Soc., Providence, RI, 1976).

	    \bibitem{tits}  Tits J., Free subgroups in linear groups, \textit{J. Algebra} \textbf{20} (1972) 250-270.
	    
	    \bibitem{wehrfritz86}  Wehrfritz B. A. F., Soluble normal subgroups of skew linear groups, \textit{J. Pure Appl. Algebra} \textbf{42} (1986) 95-107.
	    	    		
		\bibitem{wehrfritz87}  Wehrfritz B. A. F., Locally soluble skew linear groups, \textit{Math. Proc. Cambridge Philos. Soc.} \textbf{102} (1987) 421-429.
				
\end{thebibliography}
\end{document}